\documentclass{article}

\usepackage[english]{babel}
\usepackage{amsmath,amssymb,amsthm,mathrsfs}
\usepackage{mathtools}
\usepackage{url}
\usepackage{float}
\usepackage{graphicx}
\usepackage{color, comment, xcolor}
{\newtheorem{remark}{Remark}}
\newtheorem{theorem}{Theorem}
\newtheorem{lemma}{Lemma}

\definecolor{commentcolor}{RGB}{255, 165, 0}

\newcommand{\roy}[1]{\textcolor{black}{#1}}

\newcommand{\Rb}{\mathbb{R}}

\title{$L^2$-based stability of blowup with log correction for semilinear heat equation}
\date{\today}
\begin{document}
\author{Thomas Y. Hou, Van Tien Nguyen, Yixuan Wang\footnote{\noindent\textbf{Corresponding author:} Yixuan Wang (\texttt{roywang@caltech.edu})}}
\maketitle

\begin{abstract}
We propose an alternative proof of the classical result of Type-I blowup with log correction for the semilinear \roy{heat} equation. Compared with previous proofs, we use a novel idea of enforcing stable normalizations for perturbation\roy{s} around the approximate profile and \roy{we} establish a weighted $H^k$ stability, thereby avoiding the use of a topological argument and the analysis of a linearized spectrum. \roy{Consequently,} this approach can be adopted even if we only have a numerical profile and do not have explicit information on the spectrum of its linearized operator. This result generalizes the $L^2$-based stability \roy{framework beyond exactly self-similar blowup} and can be adapted to higher dimensions. Numerical results corroborate the effectiveness of our normalization, even in the large perturbation regime beyond our theoretical setting.

    \end{abstract}
\section{Introduction}
We consider the semilinear heat equation \begin{equation}
    \label{sml}
    a_t=\Delta a + a^2\,,
\end{equation}
where $a(t): \Rb^n \to \Rb$. 
\roy{We impose the decay condition} $\lim_{|x| \to \infty} a(x,t) = 0$ \roy{with} initial data $a(0) = a_0$. %\vnn{The local Cauchy problem for \eqref{sml} can be solved in Lebesgue spaces $H^m$ through a semigroup approach thanks to a Banach fixed point argument (see for example \cite{Weiiumj80} for the case of bounded domains whose proof can be entirely extended to the whole space case, see Appendix \ref{sec:appendix})} %or in the special affine space $\mathcal{H}_{s,b, k}$ for some positive constants $s, b$ and integer $k \geq 1$ of our interest, where
%\begin{equation}\label{def:Hsbk}
%\mathcal{H}_{s,b, k} \coloneqq \{ u:u = \psi_{s,b} + g,  \;\; \textup{with} \;\; \psi_{s,b}(x) = \frac{s}{1 + b|x|^2} \quad \textup{and}\;\; g \in \mathcal{E}_k \},
%\end{equation}
%and $\mathcal{E}_k$ is a weighted Sobolev space specified for our analysis, see \eqref{def:Ek} for the definition. In particular, the problem \eqref{sml} has a unique classical solution $a(t) \in \mathcal{H}_{s,b, k}$ defined in $[0,T)$ with $T \leq +\infty$ (see Appendix \ref{sec:appendix} for a proof). In the case $T = +\infty$, we have a global existence for $a(t)$, on the contrary, i.e. $T < +\infty$, the solution $a(t)$ blows up in finite time $T$, namely $\lim_{t \to T} \|a(t)\|_{L^\infty(\Rb^n)} = +\infty$.} 
Several blowup criteria were established in the past, dating back to the works of Kaplan \cite{Kapcpam63}, Fujita \cite{FUJsut66}, Levine \cite{LEVarma73},  Friedman-McLeod \cite{FMiunj85}, etc. We refer to the book \cite{quittner2019superlinear} for a comprehensive review on this subject.   Given our interest in the singularity formation, in particular in characterizing blowup solutions to \eqref{sml}, we \roy{focus on mentioning} works where a precise law of the blowup can be identified.

Singularity formation in nonlinear PDEs is \roy{often linked to symmetry groups of the underlying equation}. \roy{Equation \eqref{sml}} is invariant under the scaling  
\begin{equation}\label{def:scaling}
\forall \lambda > 0, \quad a_\lambda(x,t) = \frac{1}{\lambda^2}a\Big( \frac{x}{\lambda}, \frac{t}{\lambda^2}\Big)\,,
\end{equation}
in the sense that if $a$ is a solution to \eqref{sml}, so is $a_\lambda$ with the rescaled initial data $a_{0, \lambda} = \frac{1}{\lambda^2}a_{0}\big(\frac{x }\lambda \big)$. The earliest application of this scaling invarian\roy{ce} to equation \eqref{sml} that we are  aware of is the work by Berger-Kohn \cite{berger1988rescaling}, where the authors developed a  rescaling algorithm to capture the blowup profile 
\begin{equation}\label{stabProfile}
a(x, t) \sim \frac{1}{T-t} \bar{u}\left(\frac{x}{\sqrt{(T-t)|\log (T-t)|}}\right)\,, \quad \bar{u}(\xi)=\frac{1}{1+|\xi|^2 / 8}\,.
\end{equation}
This blowup behavior is in agreement with classification results rigorously established in the works of Filippas-Kohn \cite{filippas1992refined}, Filippas-Liu \cite{filippas1993blowup}, Herrero-Velazquez \cite{HVaihn93} and Velazquez \cite{velazquez1992higher} under the assumption of Type-I blowup, \roy{i.e.}  $$\limsup\limits_{t\to T}\|a(t)\|_{L^\infty} (T-t) < +\infty\,.$$ \roy{If the above limit is unbounded}, the blowup is of Type-II.
In particular, Herrero-Velazquez \cite{HVasnsp92} showed that the blowup behavior \eqref{stabProfile} is generic in dimension $1$, and they \roy{announced} the same for higher dimensions in an unpublished work. A rigorous construction was later established by Bricmont-Kupiainen \cite{bricmont1994universality} to provide concrete examples of initial data leading to blowup behaviors classified in the works mentioned above. The method developed in \cite{bricmont1994universality} was generalized in the work of Merle-Zaag \cite{merle1997stability} {(see also \cite{MZjems24}, \cite{NZaens17})}  through spectral analysis and a topological argument to establish the existence and stability of blowup solutions to \eqref{sml} with the asymptotic behavior \eqref{stabProfile}. 

It is worth mentioning that the scaling invariance \eqref{def:scaling} gives rise to the notion of energy-criticality in the sense that 
$$\|a_\lambda\|_{\dot{H}^1} = \lambda^{\roy{\frac{n}{2}-3}}\|a\|_{\dot{H}^1}.$$
\roy{Accordingly,} the problem \eqref{sml} is called energy-critical if $n = 6$, energy-subcritical if $n \leq 5$ and energy-supercritical if $n \geq 7$. It is well known that the only Type-I blowup occurs in the energy-subcritical case (i.e., $n \leq 5$) from the work of Giga-Kohn \cite{GKcpam85, GKiumj87, GKcpam89}, Giga-Matsui-Sasayama \cite{GMSiumj04} {(see also \cite{GMSmmas04} for the case of convex domains and Quittner \cite{Qduke21} for non-convex domains).}  The blowup in the energy-critical and supercritical cases is more subtle; in particular,  Type-II blowup may \roy{occur,} as predicted in \cite{HVcrasp94} and \cite{FHVslps00} \roy{via formal matched} asymptotic expansions. {Concrete examples of initial data leading to Type-II blowup for \eqref{sml} with a general nonlinearity $|a|^{p-1}a$ were exhibited in several works \cite{HVpre94}, \cite{Mma07}, \cite{Capde17},  \cite{CMRjams19},  \cite{PMWjfa21}, \cite{Sjfa12}, \cite{PMWam19}, \cite{PMWZdcds20}, \cite{Hihp20}, \cite{PMWZZarx20},  \cite{Hapde20} and references therein. In particular, the Type-II blowup constructed in \cite{Hapde20} for the energy-critical case in dimension $n = 6$ \roy{corresponds} to the exponent $p=2$ considered in this present work.} Partial classification of Type-II blowup was provided in  \cite{MMcpam04, MMjfa09} and \cite{CMRcmp17}, even though a complete classification of all blowup patterns still remains open. 

In this paper, we are interested in adopting the idea of dynamic rescaling to provide rigorous proofs for the semilinear heat equation with a clear notion of stability. Specifically, just like in numerical algorithms, we introduce proper rescaling conditions to ensure the stability of the perturbation around the approximate steady state, whose proof constitutes the main goal of this article.  We adopt a $L^2$-based stability analysis with properly chosen singular weights and normalization conditions, inspired by the line of work pioneered by \cite{chen2021finite, chen2019finite2}, and present our main result as Theorem \ref{thm1}. {We introduce the weighted Sobolev space $\mathcal{E}_k$ for $k = 2n + 10$: 
\begin{equation}\label{def:Ek}
\mathcal{E}_k = \Big\{u: \|u\|^2_{\mathcal{E}_k} = \| u\|^2_{\rho}+\mu\|\nabla^k u\|_{\rho_k}^2 < +\infty\Big\},
\end{equation}
where $\| \cdot\|_{\rho}$ stands for the weighted $L^2$-norm with the weight functions $\rho$, $\rho_k$ being defined in  \eqref{rho_def}, \eqref{rhok_def}, {and $\mu > 0$ is a fixed small constant (see \eqref{def_energy})}.} 

\begin{theorem}\label{thm1}
\roy{Let $k=2n+10$. There exist positive constants $C_0$ and $\lambda_0$ such that the following holds.
For any $\lambda\in(0,\lambda_0)$, if $g$ is even\footnote{We call a multivariate function $a$ even if it is even in each coordinate, i.e. $a(\xi_1 x_1,\dots,\xi_n x_n)=a(x_1,\dots,x_n)$ for all $\xi_i\in\{-1,1\}$.}
and satisfies $\|g\|_{\mathcal{E}_k}\le C_0\lambda$, then the solution to \eqref{sml} with initial data
\[
a(x,0)=\lambda^{-1}\big(\bar{u}(x)+g(x)\big)
\]
blows up in finite time $T<\infty$. Moreover, as $t\to T$ we have convergence in $\mathcal{E}_k$:
$$\lim _{t \rightarrow T}(T-t) a\left(((T-t)|\log (T-t)|)^{\frac{1}{2}} z, t\right)=\bar{u}(z).$$}
\end{theorem}
\begin{remark}
    Using the scaling invariance of \eqref{sml}, we can introduce an initial rescaling in space (corresponding to introducing $\hat{C}_l(0)$  in the dynamic rescaling formulation \eqref{cl} in Section \ref{sec:dy}) to obtain a result comparable to the theorems in \cite{bricmont1994universality,merle1997stability} that characterize the blowup time precisely. Here we highlight obtaining the correct rate and, for the sake of simplicity, do not rescale in space at $t=0$.

\end{remark}\begin{remark} We work under the even assumption for illustration purposes of the dynamical rescaling technique developed in this paper. We refer to our subsequent work \cite{chen2024stability} on the complex Ginzburg-Landau equation for a generalized dynamical rescaling technique to remove the even assumption, where translational and rotational modulations were introduced; see step 1 in Section 1.2 therein.
\end{remark}
Compared with most of the aforementioned works on semilinear equations that work in parabolic scaling, we work in variables that correspond to the true blowup scaling and obtain stability precisely with respect to the weighted $H^k$ norms we construct, instead of resorting to a topological argument to identify the existence of the initial data for blowup.
\subsection{Literature review and main contributions}
The idea of dynamic rescaling formulation or the modulation technique to study blowup was originally introduced in the numerical study of self-similar blowup of the nonlinear Schr\"odinger equation  \cite{weinstein1985modulational,soffer1990multichannel,soffer2006soliton,mclaughlin1986focusing,landman1988rate}. Later on, the formulation has been generalized to various dispersive problems, both as numerical techniques and as an analysis tool; see for example nonlinear Schr\"odinger equation \cite{kenig2006global,merle2005blow}, compressible Euler equations \cite{buckmaster2019formation},  and the nonlinear heat equation \cite{berger1988rescaling,merle1997stability}. Recently, this modulation technique has been adopted to establish self-similar singularity for incompressible Euler equations in \cite{ elgindi2021finite, chen2019finite2, chen2022stable}.

When the equation admits an analytical approximate profile for blowup, analyzing the spectrum of the linearized operator has proven to be useful for establishing the blowup in many cases; see for example, the semilinear heat equation \cite{bricmont1994universality, merle1997stability} and the 2D Keller-Segel equation \cite{raphael2014stability, collot2022refined}. While this methodology is powerful, it hinges on the fact that we are able to construct a simple and analytical approximate steady state and can analyze the spectrum of the linearized operator explicitly (for semilinear heat equations) or at least asymptotically (for Keller-Segel equations). In this paper, we provide a proof of blowup for the semilinear heat equation without analyzing the eigenvalues or the eigenfunction of the linearized operator at all, and we rule out the unstable directions via a clear characterization of a singularly weighted Sobolev space, instead of using Brouwer's fixed-point theorem and a topological argument. {This framework can be adopted even if we
only have a numerical or implicit profile and do not have explicit information on the
spectrum of its linearized operator; see the follow-up work by the last author and collaborators \cite{liu2025finite} on 3D Keller-Segel equation with logistic damping.} 

On the other hand, a direct $L^2$ \cite{chen2021finite} or $L^\infty$-based \cite{chen2022stable} stability argument with appropriate normalization conditions has been proven successful, even if no explicit approximate steady state can be identified. In fact, they are often combined with a numerical profile and rigorous computer-assisted proofs. See  \cite{chen2021finite, chen2020singularity, hou2024blowup} for applications in various 1D models for the Euler equations, \cite{ chen2019finite2, chen2022stable} for 3D axisymmetric incompressible Euler equations, and { \cite{MRRSam22a} for the compressible Navier-Stokes equation (and the accompanying paper on the compressible Euler equation \cite{MRRSam22b}). The nature of the blowup in \cite{MRRSam22b, MRRSam22a} and the current paper is quite similar in the sense that the dominant behavior of the rescaled equations is driven by the inviscid part, although in each case one must use different scaling to get the precise asymptotic behavior.} The methodology can be roughly summarized in the following two steps. Firstly, we link self-similar singularity with convergence to a steady state using the dynamic rescaling equation and obtain approximate steady states either analytically or numerically. Then, upon choosing appropriate normalization conditions, we can perform linear and nonlinear stability estimates to show that the perturbation around the approximate steady state remains small. Therefore, we can obtain an asymptotically self-similar blowup with rates prescribed by the normalizing constants.

%Up until now, this line of work has been somewhat limited to the self-similar setting \yw{to the authors' awareness} since it was believed that one has to at least; maybe except for the work. 
This article adopts the $L^2$-based methodology to establish blowups beyond the self-similar setting; see also \cite{collot2023stable} on the 1D inviscid primitive equation and \cite{van2003formal} on the harmonic map heat flow, where log corrections were also observed. We show that one can obtain the correct blowup rate by imposing proper vanishing conditions on the perturbation, without a priori knowledge of a formal blowup rate. Compared with a self-similar blowup, the crucial difference is that now we have an algebraic, instead of exponential, convergence of the normalizing constants in the rescaled time $\tau$, inferred for example, by \eqref{cu}. 

Another contribution is that we introduce different spatial rescalings in $n$ different dimensions in Section \ref{highd}, giving enough degrees of freedom for the normalization conditions. Those different rescaling constants in different dimensions will indeed converge to the same rescaling constant close to the blowup time. This approach may shed some light on the generalization of the dynamic rescaling framework to higher dimensions for other problems. We believe our method is robust for Type-I singularities, especially for problems having non-self-adjoint linear operators, see for example our follow-up work  \cite{chen2024stability} on the complex Ginzburg-Landau equation.
 Finally, we demonstrate our choice of normalization to be effective even beyond the regime of small perturbations, in Section \ref{sec:num} based on numerical experiments. 

 \subsection{Notations}
Throughout the article, we use $(\cdot,\cdot)$ to denote the inner product on $L^2(\mathbb{R}^n)$: $(f,g)=\int_{\mathbb{R}^n}fgdx$.     We use $C$ to denote absolute constants dependent only on the dimension $n$, which may vary from line to line. We use $A\lesssim B$ for positive $B$ to denote that there exists a constant $C>0$ such that $A\leq CB$.  We adopt the notation of the Japanese bracket as $\langle z\rangle=\sqrt{1+|z|^2}\,.$
\section{Dynamic rescaling formulation and normalization conditions}
\label{sec:dy}
In this section, we \roy{present} our dynamic rescaling formulation. Via enforcing local vanishing modulation conditions, we derive the law of the blowup formally and motivate the choices of singular weights \roy{used in the} stability analysis.
\subsection{1D case}
We focus on the 1D case first and generalize to higher dimensions in Section \ref{highd}.
For the semilinear heat equation \eqref{sml}, we introduce the dynamic rescaling formulation
\begin{equation*} \label{def:renormua}
\hat{u}(z,\tau)=\hat{C}_u(\tau) a(\hat{C}_l(\tau)z,t(\tau))\,,
\end{equation*}
where 
\begin{equation}\label{def:CuCl}
\hat{C}_u=\hat{C}_u(0)\exp{\Big(\int_0^\tau \hat{c}_u(s) ds\Big)},\;  \hat{C}_l=\exp{\Big(-\int_0^\tau\hat{c}_l(s) ds\Big)},\; t=\int_0^\tau \hat{C}_u(s) ds,
\end{equation}with $\hat c_u$ and $\hat c_l$ being determined. We introduced the extra degree of freedom $\hat{C}_u(0)$ as in \cite{chen2020singularity, hou2024blowup}, which we will later choose to be small \roy{to control} the viscous term.
The renormalized equation for $\hat u$ reads 
\begin{equation}\label{eq:uhatztau}
    \hat{u}_\tau=\hat{c}_u \hat{u}-\hat{c}_l z\hat{u}_z+\hat{u}^2+\frac{\hat{C}_u}{\hat{C}_l^2}\hat{u}_{zz}.
\end{equation}

    As we shall show later, $\frac{\hat{C}_u}{\hat{C}_l^2}$ goes to zero as $\tau \to \infty$.  \roy{We} consider the approximate \roy{steady state} $\bar{u}=(1+z^2/8)^{-1}$ \roy{of the inviscid part of} \eqref{eq:uhatztau}, namely
    $$ \bar{c}_u\bar{u}-\bar{c}_l z\bar{u}_z+\bar{u}^2=0, \quad \textup{for}\;\; \bar{c}_u=-1\,,\quad\bar{c}_l=1/2\,.$$
 By using the dynamic rescaling formulation, we reduce the problem of establishing a blowup in the physical variables and quantifying its blowup rate to the problem of establishing stability \roy{of the perturbation around this approximate steady state $\bar{u}$,} in the dynamic rescaling formulation {\eqref{eq:uhatztau}}. We want to show that $\hat{u}$ converges to the steady state $\bar{u}$ of the dynamic rescaling equation, and the normalization constants also converge.
We {consider} the ansatz  $$\hat{u}=\bar{u}+u, \quad  \hat{c}_u=\bar{c}_u+c_u, \quad  \hat{c}_l=\bar{c}_l+c_l\,.$$
We will \roy{explain how to choose $c_u$ and $c_l$ via normalization conditions so that the perturbation $u$ remains small for all $\tau\ge 0$, which in turn yields the correct blow-up scaling.}

 {If we enforce that $u$  is an even function satisfying $u(0, \tau) = u_{zz}(0, \tau) = 0$ for all time $\tau$, we have by the dynamic rescaling equation \eqref{eq:uhatztau},}
\begin{equation*}\label{cons_hat}
    \hat{c}_u+\bar{u}(0)+\frac{\hat{C}_u \bar{u}_{zz}(0)}{\hat{C}_l^2 \bar{u}(0)}=0\,,\hat{c}_u-2\hat{c}_l+2\bar{u}(0)+\frac{\hat{C}_u (\bar{u}_{zzzz}(0)+{u}_{zzzz}(0))}{\hat{C}_l^2 \bar{u}_{zz}(0)}=0\,.
\end{equation*}

 Define 
$$\lambda (\tau)=\frac{\hat{C}_u (\tau)}{\hat{C}_l^2(\tau)}=\hat{C}_u(0)\exp{\Big(\int_0^\tau ({c_u(s)}+2c_l(s)) ds\Big)}\,.$$
We can simplify the normalization constraints into $${c_u}-\frac{1}{4}\lambda=0\,,\quad{c_u}-2{c_l}-\big(\frac{ 3}{2}+4{u}_{zzzz}(0)\big)\lambda=0\,.$$
This gives 
\begin{equation}\label{cu}
{c_u}=\frac{1}{4}\lambda\,, \quad {c_l}=-(\frac{ 5}{8}+2{u}_{zzzz}(0))\lambda\,,\end{equation}
from which we {derive} the ODE for $\lambda$ as 
\begin{equation}\label{lamb}
    \lambda_\tau=\lambda(c_u+2c_l)=-(1+4{u}_{zzzz}(0))\lambda^2\,.
\end{equation}
\begin{remark}
    \roy{When the perturbation $u_{zzzz}(0)$ is small, the ODE \eqref{lamb} suggests} $$\lambda\approx \frac{1}{\tau}  \approx \frac{1}{|\log{(T-t)}|}\,.$$ Therefore, the factor $\frac{\hat C_u}{\hat C_l^2} \approx \frac{1}{\tau}$, \roy{and}  the effect of the viscosity term in \eqref{eq:uhatztau} can be treated perturbatively. We will make this heuristic argument rigorous later on by choosing $\hat{C}_u(0)$ small.
\end{remark}
%We  
\begin{remark}\label{rmk2}
\roy{To motivate our normalization conditions, we consider a singular weight $\rho(z)\sim |z|^{-\alpha}$ near $z=0$ in the $L^2$ energy estimate. A formal integration by parts suggests that, near the origin, the linear terms yield damping of the form
\[
(u_\tau,u\rho)\approx\Big(-1+\frac14\frac{(\rho z)_z}{\rho}+2\Big)(u,u\rho)
=\Big(1-\frac{\alpha-1}{4}\Big)(u,u\rho).
\]
Thus we need $\alpha>5$ to obtain linear damping, which in turn motivates enforcing that $u$ vanishes at $z=0$ up to third order (so that the weighted energy is finite and coercive). Since we also need to control nonlinear terms, the actual weights we use cannot be as simple as $\rho(z)=|z|^{-6}$, but this provides a useful starting point for the stability analysis.}
\end{remark}

\subsection{nD case}\label{highd}
{The crucial idea} in the $n$-dimensional case is \roy{to} introduce $n$ \roy{distinct spatial scaling parameters, one in each coordinate direction. This provides enough freedom to impose normalization conditions that enforce the same vanishing orders at the origin.} Consider
\begin{equation}
    \label{scaling:nd}\hat{u}(z,\tau)=\hat{C}_u(\tau) a(\hat{C}_{l,1}(\tau)z_1,\hat{C}_{l,2}(\tau)z_2,\cdots,\hat{C}_{l,n}(\tau)z_n,t(\tau))\,,\end{equation} with the same $\hat{C}_u$ and $t(\tau)$ {introduced in \eqref{def:CuCl}} and 
    \begin{equation}
    \label{cl}\hat{C}_{l,i}(\tau)=\exp{\Big(\int_0^\tau -\hat{c}_{l,i}(s) ds\Big)}\,.
\end{equation}\roy{Define
$$\lambda_i(\tau)=\frac{\hat{C}_u(\tau)}{(\hat{C}_{l,i}(\tau))^2}=\hat{C}_u(0)\exp{\Big(\int_0^\tau \big({c_u(s)}+2c_{l,i}(s)\big) ds\Big)}\,.$$}
The equation for $\hat{u}$ {reads as} 
\begin{equation}\label{drf}
\hat{u}_\tau=\hat{c}_u \hat{u}-\sum_i\hat{c}_{l,i}z_i \hat{u}_{i}+\hat{u}^2+\sum_i\lambda_i\hat{u}_{ii}\,,
\end{equation}
where we use the short-hand notation for partial derivatives: we denote $f_{i}=\partial_{z_i}f$ and $f_{ij}=\partial_{z_j}\partial_{z_i}f$.
Using the same radial approximate steady state $\bar{u},\bar{c}_l,\bar{c}_u$ and a similar ansatz
\begin{equation}
\label{ansatz}\hat{u}=\bar{u}+u\,,\quad \hat{c}_u=\bar{c}_u+c_u\,, \quad\hat{c}_{l,i}=\bar{c}_l+c_{l,i}\,,
\end{equation}
we can enforce the same normalization condition that {$u$ is of $O(|z|^4)$ as $|z| \to 0$}. Notice that if we choose  $u$ to be an even perturbation, we only need to enforce {$u(0, \tau)= {u}_{ii}(0, \tau)=0$. These} $n+1$ constraints can be solved exactly to obtain 
\begin{equation}
\label{cun}
    {c_u}=\frac{1}{4}\sum_i\lambda_i\,, \quad {c_{l,i}}=-\sum_j\lambda_j\big(\frac{ 1+4\delta_{ij}}{8}+2{u}_{iijj}(0)\big)\,,
\end{equation}
where $\delta_{ij}=1$ if $i=j$, and $0$ otherwise.
We obtain the ODE {for $\lambda_i$,}
\begin{equation}
\label{lambn}
\partial_\tau\lambda_i=\lambda_i(c_u+2c_{l,i})=-\big(\sum_j4{u}_{iijj}(0)\lambda_j+\lambda_i\big)\lambda_i\,.
\end{equation}
Notice that  \eqref{cun}, \eqref{lambn} are consistent with  \eqref{cu}, \eqref{lamb} derived for the 1D case.
\begin{remark}
The vanishing conditions that $u=O(|z|^4)$ near the origin for the even perturbation can be guaranteed by the dynamic constraint in \eqref{cun}. On the other hand, it is also equivalent to an appropriate rescaling of $a$ at each time $\tau(t)$. To be precise, it amounts to $\hat{u}(0,\tau)=\hat{u}(0,0)$ and $\hat{u}_{ii}(0,\tau)=\hat{u}_{ii}(0,0)$, and by \eqref{scaling:nd}, it is equivalent to
    \begin{equation}
\label{32}\hat{C}_u(\tau)a(0,t)=\hat{C}_u(0)a(0,0),\quad \hat{C}_u(\tau)(\hat{C}_{l,i}(\tau))^2a_{ii}(0,t)=\hat{C}_u(0)a_{ii}(0,0)\,.
 \end{equation}This formulation will be useful in the local well-posedness in Appendix \ref{sec:appendix}.
\end{remark}
 \section{Stability of perturbation and blowup}\label{sec3}
 Building upon the general strategy of a weighted $L^2$-based stability argument as in \cite{chen2021finite, chen2019finite2}, we will prove Theorem \ref{thm1} in this section. We \roy{set} $\lambda=\max_i{\lambda_i}$.

\subsection{$L^2$ stability analysis}\label{subl2}

Plugging in the ansatz \eqref{ansatz} into the dynamic rescaling equation \eqref{drf} and using the fact that $\bar{u}$ is an approximate steady state, we write down the evolution equation for $u$ as follows:
\begin{equation}\label{eq:uztau}
    {u}_\tau=L(u)+N(u)+\sum_i F_i(z, \tau)+ \sum_i \lambda_i V_i(u)\,,
\end{equation}
where we reorganize the different terms into the linear, nonlinear, error, and viscous terms respectively as $$\quad L=(-1+{c_u}) {u}- \sum_i(\frac{1}{2}+{c_{l,i}})z_i {u}_{i}+2\bar{u}u\,,\quad N={u}^2\,,$$
$$ F_i=\frac{1}{4}\lambda_i\bar{u}-{c_{l,i}}z_i\bar{u}_i+\lambda_i (\bar{u}_{ii}+\sum_j\frac{1}{2}u_{iijj}(0)z_j^2\chi(|z|))\,,$$
$$ V_i={u}_{ii}-\sum_j\frac{1}{2}u_{iijj}(0)z_j^2\chi(|z|)\,.$$
\roy{Here $\chi=\chi(r)$ is a smooth cutoff on $[0,\infty)$ with $\chi(r)=1$ for $r\le1$ and $\chi(r)=0$ for $r\ge2$.} We introduce such a cutoff function to make each one of the four terms integrable in the weighted $L^2$ space. We name and group the terms in such a way that is convenient for our analysis. The ``linearized operator $L$'' is obtained by treating the scaling parameters $c_u$ and $c_l$ as known parameters, although they actually depend on $u$. As a result, $L$ actually contains nonlinear terms in the original physical variables.

To show that the dynamic rescaling equation is stable for even perturbations and converges to a steady state, we will perform {an energy estimate in $L^2_\rho$: the weighted $L^2$ space with the weight and the associated norm defined by} 
\begin{equation}
\label{rho_def}\rho=|z|^{-5-n}+10^{-3}|z|^{1-n}\,,\quad \|f\|_{\rho}=(f^2,\rho)^{1/2}\,.
\end{equation}
We choose such a weight to extract damping near the origin as {mentioned in} Remark \ref{rmk2}, while also having 
good control of growth at infinity, handling the nonlinear estimates easier. 

Via an integration by parts\footnote{\roy{The integrations by parts here (and below) are justified by cutoff and mollification. Fix $\varepsilon\in(0,1)$ and choose a smooth radial cutoff $\eta_\varepsilon$ supported in the annulus $\{\varepsilon\le |z|\le 1/\varepsilon\}$ with $\eta_\varepsilon\equiv1$ on $\{2\varepsilon\le |z|\le 1/(2\varepsilon)\}$. Since $\rho(z)\sim |z|^{-n-5}$ near $0$ and $\rho(z)\sim |z|^{1-n}$ as $|z|\to\infty$, and since $u(z,\tau)=O(|z|^4)$ at the origin and $u\in L^2_\rho$ at infinity, we have $\|\eta_\varepsilon u-u\|_{L^2_\rho}\to0$ as $\varepsilon\to0$ (and similarly for the derivatives appearing in the identities). On the annulus, the weighted norms are equivalent to unweighted $L^2$ norms, so we can further approximate $\eta_\varepsilon u$ by $C_c^\infty$ functions via convolution; for such smooth compactly supported approximants, the boundary terms vanish. Passing to the limit using Cauchy--Schwarz yields the stated identities.}}, we have the standard $L^2$ estimate for the linear part:
\begin{equation*}
    (L,u\rho)=([(-1+{c_u})+\frac{1}{2} \sum_i(\frac{1}{2}+{c_{l,i}})\frac{(z_i\rho)_i}{\rho}+2\bar{u}]u,u\rho)\,.
\end{equation*}
We plug in the singular weight \eqref{rho_def}, using $O(\lambda)$ notations due to the form \eqref{cun}, and simplify as
$$\begin{aligned}
    &(-1+{c_u})+\frac{1}{2} \sum_i(\frac{1}{2}+{c_{l,i}})\frac{(z_i\rho)_i}{\rho}+2\bar{u}\\=&O((1+|\nabla^4u(0)|)\lambda)-\frac{1}{4}+\frac{0.006}{4 (10^{-3}+\roy{|z|}^{-6})}-\frac{2\roy{|z|}^2}{8+\roy{|z|}^2}\,.
\end{aligned}$$
By a straightforward computation and the AM-GM inequality we have $$0.006(8+\roy{|z|}^2)-4 (10^{-3}+\roy{|z|}^{-6})2\roy{|z|}^2=0.048-0.002\roy{|z|}^2-8\roy{|z|}^{-4}\leq 0\,.$$
Therefore, we have the simple linear stability
\begin{equation*}
    (L,u\rho)\leq(-\frac{1}{4}+O((1+|\nabla^4u(0)|)\lambda))(u,u\rho)\leq(-\frac{1}{4}+C(1+|\nabla^4u(0)|)\lambda)\|u\|_{\rho}^2\,.
\end{equation*}

The estimate of the nonlinear term is straightforward:
$$(N,u\rho)\leq\|u\|_{\infty}\|u\|_{\rho}^2\,.$$

{We now estimate the generated error term. A direct computation yields} $\bar{u}_{ii}=-\frac{\bar{u}^2}{4}+\frac{z_i^2\bar{u}^3}{8}$. As a consequence, we use Fubini's \roy{theorem} to obtain 
$$\begin{aligned}
\sum_{i}F_i&=\frac{\sum_i\lambda_i}{4}(\bar{u}+\frac{1}{2}z\cdot\nabla\bar{u}-\bar{u}^2)+\sum_i\frac{\lambda_i}{2}(z_i\bar{u}_i+\frac{z_i^2\bar{u}^3}{4})\\&+\sum_j\sum_i\frac{\lambda_i u_{iijj}(0)}{2}(4z_j\bar{u}_j+z_j^2\chi(|z|))\,.
\end{aligned}$$Using the identity $\bar{u}+\frac{1}{2}z\cdot\nabla\bar{u}-\bar{u}^2 = 0$, the first term vanishes. We then simplify the remaining terms by using $z_i\bar{u}_{i}=-z_i^2\frac{\bar{u}^2}{4}$ to get 
\begin{equation*}
    \sum_{i}F_i=-\sum_i\lambda_iz_i^2|z|^2\frac{\bar{u}^3}{64}+\sum_j\sum_i\frac{\lambda_i u_{iijj}(0)z_j^2}{2}(\chi(|z|)-\bar{u}^2)\,.
\end{equation*}
We hence know that the error term is $O(|z|^4)$ at $z=0$ and $O(|z|^{-2})$ at $\infty$. \roy{Therefore,  we have {$\sum_{i}F_i\in L^2_\rho$}}. We conclude that {$$(\sum_{i}F_i,u\rho)\leq C(1+|\nabla^4u(0)|)\lambda\|u\|_{\rho}\,.$$}

The viscous part is more subtle and needs to {be handled carefully, \roy{due to} the singular weight $\rho$}. We recall from \eqref{rho_def} the estimate $$|\nabla^2\rho|\lesssim \rho/|z|^2\,.$$
We do integration by parts twice to get
\begin{equation*}
\begin{aligned}
    (V_i,u\rho)&=(-\frac{1}{2}u_{iijj}(0)\roy{z_j^2\chi(|z|)}|z|,\frac{u}{|z|}\rho)-(u_i,u_i\rho)-(u_i,u\rho_i)\\&\leq-\|u_i\|_{\rho}^2+C(|\nabla^4u(0)|^2+\|\frac{u}{|z|}\|_{\rho}^2)\,.
\end{aligned}
\end{equation*}
 {To estimate the last term, we split the integral into $|z| \leq 1$ and $|z| \geq 1$ and use the fact that $u(z,\tau) = O(|z|^4)$ for $|z| \ll 1$ to obtain
$$\|\frac{u}{|z|}\|_{\rho}^2\lesssim\int_{|z| \leq 1}\frac{u^2}{|z|^{7+n}}+\int_{|z| \geq 1}{u^2}{\rho}\lesssim(\sup_{|z| \leq 1}\frac{|u|}{|z|^4})^2+\|u\|_{\rho}^2\lesssim\|\nabla^4u\|_{\infty}^2+\|u\|_{\rho}^2\,.$$
}

\roy{Denoting} $E^2_0(u)=(u,u\rho)$ \roy{and collecting} all the above estimates, we arrive at  the weighted-$L^2$ estimate
\begin{align}\label{L2-col}
    &\frac{1}{2}\partial_\tau E^2_0=(L+N+\sum_i(F_i+\lambda_iV_i),u\rho) \\
    & \leq(-\frac{1}{4}+C(1+\|\nabla^4u\|_{\infty})\lambda+\|u\|_{\infty})E^2_0 +C\lambda(1+\|\nabla^4u\|_{\infty}) E_0+C\lambda\|\nabla^4u\|_{\infty}^2\,.\nonumber
    \end{align}
To close the weighted-$L^2$ estimate, we need higher-order estimates that we perform in the next subsection.
\subsection{Higher order stability analysis}
\label{hknorm}
Consider the weighted $H^k$ norm for $k=2n + 10$:
\begin{equation}\label{rhok_def}
    E^2_k(u)=(\nabla^ku,\nabla^ku\rho_k)\,,\quad \rho_k=1+10^{-3k}|z|^{2k+1-n}\,,
\end{equation} 
and {we write from \eqref{eq:uztau} the energy identity}
\begin{equation*}
    \frac{1}{2}\partial_\tau E^2_k=(\nabla^kL+\nabla^kN+\sum_i(\nabla^kF_i+\lambda_i\nabla^kV_i),\nabla^ku\rho_k)\,.
\end{equation*}
We will need the following lemma concerning interpolation inequalities of lower order terms and $L^\infty$ estimates of a Morrey-type. We define the weighted auxiliary norms 
\begin{equation}
\|u\|_{Q_j}=\|\nabla^j u\langle z\rangle^{j+(1-n)/2}\|_{2}.\label{Dj}
\end{equation}

By \eqref{rho_def} and \eqref{rhok_def}, we know that $\|u\|_{Q_0}\lesssim E_0$ and $\|u\|_{Q_k}\lesssim E_k$.
\begin{lemma}\label{lem:int}
    For any $\nu>0$, there exists $C = C(\nu) > 0$ such that 
    $$\|u\|_{Q_j}\leq\nu \|u\|_{Q_k}+C(\nu)\|u\|_{Q_0}\,,\quad 0\leq j<k\,,$$
    $$\|\nabla^j u\langle z\rangle^{i+1/2}\|_{\infty}\lesssim\|\nabla^{j+n}{u}\langle z\rangle^{i+({n+1})/{2}}\|_2,\quad i, j \geq 0\,.$$
\end{lemma}
\begin{proof}
    We use integration by parts to compute for $k>j>0$: $$\|u\|_{Q_j}^2=-\sum_i\int(\partial_i^2\nabla^{j-1}u\cdot\nabla^{j-1}u\langle z\rangle^{2j+1-n}+\partial_i\nabla^{j-1}u\cdot\nabla^{j-1}u\partial_i\langle z\rangle^{2j+1-n})\,.$$
    Noticing that $(\partial_i\langle z\rangle^{2j+1-n})^2\lesssim\langle z\rangle^{2j+1-n}\langle z\rangle^{2j-1-n}$, by Cauchy-Schwarz inequality, we have 
    $\|u\|_{Q_j}^2\lesssim \|u\|_{Q_{j-1}}(\|u\|_{Q_{j+1}}+\|u\|_{Q_j})$.
    By a weighted AM-GM inequality, we compute for any $\nu>0$, $\|u\|_{Q_j}^2\leq\nu \|u\|_{Q_{j+1}}^2+C(\nu)\|u\|_{Q_{j-1}}^2$.

    Now we prove the first interpolation inequality. Since $\nu$ is arbitrary, we only need to prove for $j=k-1$, which we can use induction on $k$ and the obtained inequality $\|u\|_{Q_{j}}\leq\nu \|u\|_{Q_{j+1}}+C(\nu)\|u\|_{Q_{j-1}}$ to conclude.

    For the second inequality, we borrow the idea of proof for the embedding (3.13) of Proposition 1  in our follow-up paper \cite{chen2024stability}. We can assume $u\in C^{\infty}_c$ by a density argument and consider WLOG $z\in\mathbb{R}^n$ with $z_i\geq0$ for all components. In the region $\Omega(z)=\{y\in\mathbb{R}^n,y_i\geq z_i\}$, we have $|y|\geq |z|$ for any $y\in\Omega(z)$. By Leibniz's rule and  Cauchy-Schwarz's inequality, we have 
$$\begin{aligned}|\nabla^j{u}(z)|&\lesssim \int_{\Omega(z)}|\partial_1\partial_2\cdots\partial_n\nabla^j{u}(y)|dy\\&\lesssim \|\nabla^{j+n}{u}\langle y\rangle^{i+{(n+1)}/{2}}\|_2(\int_{|y|\geq |z|}\langle y\rangle^{-2i-n-1}dy)^{1/2}\,.\end{aligned}$$
    We thus collect the pointwise bound and conclude the proof of the inequality.
\end{proof}
We denote the terms as lower order terms (l.o.t. for short) if their $\rho_k$-weighted $L^2$-norms are bounded by $\nu E_k+C(\nu)E_0$ for any $\nu>0$. Notice that for $0<j\leq k$, $\nabla^{j}\bar{u}\nabla^{k-j}u$ are l.o.t. since we can estimate $$|\nabla^{j}\bar{u}|\rho_k^{1/2}\lesssim\langle z\rangle^{-2-j}\langle z\rangle^{k+(1-n)/2}\lesssim\langle z\rangle^{k-j+(1-n)/2}\,,$$
and thus {their $L^2_{\rho_k}$-norms} are bounded by $\|u\|_{Q_{k-j}}$.
Therefore, we collect the linear estimate via an integration by parts and $O(\lambda)$ notations as 
$$\begin{aligned}(\nabla^kL,\nabla^ku\rho_k)&\leq ([-1-\frac{k}{2}+\frac{1}{4}\frac{(z\rho_k)_z}{\rho_k}+2\bar{u}]\nabla^ku,\nabla^ku\rho_k)\\&+C(1+\|\nabla^4u\|_{\infty})\lambda E_k^2+\nu E_k^2+C(\nu)E_0^2\,.\end{aligned}$$ We can compute the damping as $$-1-\frac{k}{2}+\frac{n}{4}+\frac{1}{4}\frac{(2k+1-n)10^{-3k}|z|^{2k+1-n}}{1+10^{-3k}|z|^{2k+1-n}}+\frac{2}{1+|z|^2/8}\leq-\frac{1}{2} \,,$$
 where the last inequality is equivalent to $$(1+|z|^2/8)(2k+2-n+10^{-3k}|z|^{2k+1-n})-8(1+10^{-3k}|z|^{2k+1-n})\geq0\,,$$ which can be implied by an AM-GM inequality via $$\begin{aligned}
     &3n+\frac{10^{-3k}}{8}|z|^{2k+3-n}\geq(2k+3-n)(\frac{3n}{2})^{\frac{2}{2k+3-n}}(\frac{10^{-3k}}{8(2k+1-n)})^{1-\frac{2}{2k+3-n}}|z|^{2k+1-n}\\&>10^{-3k}/8(10^{6k/(2k+3-n)})|z|^{2k+1-n}>8\times10^{-3k}|z|^{2k+1-n}\,.
 \end{aligned}$$
We collect the linear estimate by choosing a small enough $\nu$ to get
\begin{equation*}
(\nabla^kL,\nabla^ku\rho_k)\leq (-\frac{1}{4}+C(1+\|\nabla^4u\|_{\infty})\lambda)E_k^2+CE_0^2\,.
\end{equation*}

For the nonlinear term $\nabla^kN$, by Leibniz's rule, we know that it will be a linear combination of $\nabla^{k-j}u\nabla^{j}u$. For a typical term, assume WLOG that $j\leq k/2$. By the interpolation lemma, we have 
\begin{equation}
    \|\nabla^{k-j}u\nabla^{j}u\|_{\rho_k}\leq \|u\|_{Q_{k-j}}\|\nabla^ju\langle z\rangle^j\|_{\infty}\lesssim \|u\|_{Q_{k-j}}\|u\|_{Q_{j+n}}\lesssim(\|u\|_{Q_{k}}+\|u\|_{Q_{0}})^2.\label{nonlinear hk}
\end{equation}
Therefore, we can collect the nonlinear estimate 
\begin{equation*}
(\nabla^kN,\nabla^ku\rho_k)\leq C(E_k+E_0)^2 E_k\,.
\end{equation*}

For the error term, we notice that $\nabla^kF_i$ is  $O(|z|^{-2-k})$ {for $|z| \gg 1$}. Therefore, it is square integrable with the weight $\rho_k$ and we can estimate  \begin{equation*}
(\nabla^kF_i,\nabla^ku\rho_k)\leq C\lambda(1+\|\nabla^4u\|_{\infty}) E_k\,.
\end{equation*}

We estimate the viscous term using integration by parts twice
\begin{align*}
(\nabla^kV_i,\nabla^ku\rho_k)&= - (\nabla^ku_i,\nabla^ku_i\rho_k)-(\nabla^ku_i,\nabla^ku(\rho_k)_i)\\
& \leq\frac{1}{2}(\nabla^ku,\nabla^ku(\rho_k)_{ii})\leq CE_k^2\,.
\end{align*}

Finally, we collect all the above estimates and arrive at 
\begin{equation}\label{eq-hkc}
\begin{aligned}
    &\frac{1}{2}\partial_\tau E^2_k\leq (-\frac{1}{4}+C(1+\|\nabla^4u\|_{\infty})\lambda)E_k^2+CE_0^2+C(E_k+E_0)^2 E_k\\&+C\lambda(1+\|\nabla^4u\|_{\infty}) E_k+C\lambda E_k^2\,.
    \end{aligned}
\end{equation}
Using again the interpolation inequality in Lemma \ref{lem:int}, we know that $\|\nabla^4u\|_{\infty}+\|u\|_{\infty}\lesssim E_k+E_0$.
\roy{Combined with \eqref{L2-col} and \eqref{eq-hkc}, we know that there exists a sufficiently small constant $\mu>0$, such that for}
\begin{equation}\label{def_energy}
    E^2=E_0^2+\mu E_k^2\,,
\end{equation}
we end up with the estimate 
\begin{equation*}
    \frac{1}{2}\partial_\tau E^2\leq(-\frac{1}{10}+C(1+E)\lambda)E^2+CE^3+C\lambda (1+E)E+C\lambda E^2\,.
\end{equation*}
Hence, we have the following energy estimate \begin{equation}\label{energy}\partial_\tau E\leq(-\frac{1}{10}+CE\lambda+CE)E+C\lambda +C\lambda E\,,\end{equation}
for some universal constant $C > 0$. \roy{We will combine estimate \eqref{energy} with the estimates of $\lambda$ to conclude the bootstrap argument of finite time blowup in the next subsection.} 
\subsection{Finite time blowup}\label{sec4}
\roy{
Recall that for each $i \in \{1, \cdots, n\}$, $\lambda_i(\tau)$ solves the ODE \eqref{lambn}, hence $\lambda_i\in C^1$ in $\tau$. 
Define
\[
\lambda(\tau)=\max_{1\le i\le n}\lambda_i(\tau),\qquad \gamma(\tau)=\frac{1}{\lambda(\tau)}.
\]
Since $\lambda$ is the maximum of finitely many $C^1$ functions, it is locally Lipschitz; hence by Rademacher's theorem it is differentiable for almost every $\tau\ge0$ and absolutely continuous. At times \(\tau>0\) where \(\lambda(\tau):=\max_{1\le i\le n}\lambda_i(\tau)\) is differentiable, the following holds. Either the maximizer is unique, in which case there exists a neighborhood of \(\tau\) on which \(\lambda=\lambda_i\) and hence \(\lambda_\tau=(\lambda_i)_\tau\). Or multiple branches are active, i.e.\ \(I(\tau):=\{i:\lambda_i(\tau)=\lambda(\tau)\}\) contains more than one index; in that case, differentiability of the pointwise maximum forces all active branches to have the same derivative at \(\tau\). In fact, for any $i\in I(\tau)$ and $\tau_1\to\tau^-$, since we have $\lambda(\tau_1)\geq\lambda_i(\tau_1)$, $\lambda(\tau)=\lambda_i(\tau)$, and $\tau_1-\tau<0$,  we obtain 
$$\partial_\tau\lambda(\tau)=\lim_{\tau_1\to\tau^-}\frac{\lambda(\tau_1)-\lambda(\tau)}{\tau_1-\tau}\le\lim_{\tau_1\to\tau^-}\frac{\lambda_i(\tau_1)-\lambda_i(\tau)}{\tau_1-\tau}=\partial_\tau\lambda_i(\tau).$$
Taking the minimum yields $\partial_\tau\lambda(\tau)\le\min_{i\in I(\tau)}\partial_\tau\lambda_i(\tau)$. Similarly, taking limits from the right yields
$$\partial_\tau\lambda(\tau)=\lim_{\tau_1\to\tau^+}\frac{\lambda(\tau_1)-\lambda(\tau)}{\tau_1-\tau}\ge\max_{i\in I(\tau)}\partial_\tau\lambda_i(\tau).$$Hence, we conclude $\max_{i\in I(\tau)}\partial_\tau\lambda_i(\tau)=\min_{i\in I(\tau)}\partial_\tau\lambda_i(\tau)$, and one may choose any active index \(i(\tau)\in I(\tau)\) to obtain
\[
\lambda_\tau(\tau)=(\lambda_{i(\tau)})_\tau(\tau)
\qquad\text{for a.e.\ }\tau.
\]}
\roy{
Therefore, for almost every $\tau$,
\[
\gamma_\tau(\tau)=-\frac{\lambda_\tau(\tau)}{\lambda(\tau)^2}
=-\frac{(\lambda_{i(\tau)})_\tau(\tau)}{\lambda_{i(\tau)}(\tau)^2}.
\]
Using \eqref{lambn} for $\lambda_{i(\tau)}$ yields, for a.e.\ $\tau$,
\[
\gamma_\tau(\tau)
=1+4\sum_{j=1}^n u_{i(\tau)i(\tau)jj}(0,\tau)\,\frac{\lambda_j(\tau)}{\lambda_{i(\tau)}(\tau)}.
\]
Since $\lambda_j(\tau)\le \lambda(\tau)=\lambda_{i(\tau)}(\tau)$, we obtain the bound (for a.e.\ $\tau$)
\[
|\gamma_\tau(\tau)-1|
\le 4\sum_{j=1}^n |u_{i(\tau)i(\tau)jj}(0,\tau)|\frac{\lambda_j(\tau)}{\lambda_{i(\tau)}(\tau)}
\le 4n\,|\nabla^4 u(0,\tau)|.
\]}

The interpolation inequality in Lemma \ref{lem:int} yields $|\nabla^4u\roy{(0,\tau)}|\leq CE$, where WLOG, we assume $C\geq 1$. Defining $G=E\gamma$, we have, \roy{almost everywhere, that}\begin{equation}\label{odegamma}|\gamma_\tau -1|\leq 4nC\frac{G}{\gamma}\,,\end{equation}

We will show that $E$ decays like $1/\tau$  \roy{via deriving an ODE} of $G$ from \eqref{energy}:
\begin{equation}\label{odeg}\begin{aligned}
    G_\tau &\leq (-\frac{1}{10}+CE\lambda+CE)G+C+C E+E(1+4nC\frac{G}{\gamma})\\&\leq(-\frac{1}{10}+2C\frac{1}{\gamma})G+C+8nCG^2(\frac{1}{\gamma}+\frac{1}{\gamma^2})\,.
\end{aligned}\end{equation}
{We choose $\hat{C}_u(0)= 1/\gamma(0)$ small enough, for example $\hat{C}_u(0) \leq1/(10000nC^2)$} such that if we start from $G(0)< 100C$, we will have the bootstrap estimate \begin{equation}
    \label{bootstrapbound}G< 100C,\quad \gamma\geq \gamma(0)\geq10000nC^2,
\end{equation} for all time via a standard bootstrap argument {that we provide the proof below.}
\begin{proof}[Proof of the bootstrap bound \eqref{bootstrapbound}] 
We note from \eqref{odegamma} that $\partial_\tau\gamma(0) > 0$. By local wellposedness that will be established in Appendix \ref{sec:appendix}, the solution exists for a short time. {Assume that the estimate \eqref{bootstrapbound} is false}, namely that, there exists a rescaled time $\tau_0>0$ such that \eqref{bootstrapbound} holds for $0<\tau<\tau_0$ and either $G(\tau_0)\geq 100C$ or $\gamma(\tau_0)\leq \gamma(0)$. We compute by \eqref{odegamma} that in $(0,\tau_0)$, $\partial_\tau\gamma\geq1-\frac{400nC^2}{10000nC^2}>0$ which rules out the latter case. As a consequence, we estimate by the ODE for $G$ given in \eqref{odeg}: for  $\tau \in (0,\tau_0)$,  $$G_\tau \leq -\frac{1}{20} G+C+\frac{G}{50}\,.$$ By continuity, we estimate at the time $\tau = \tau_0$, $\partial_\tau G(\tau_0)<0 $ which concludes that the former case cannot hold either. Hence, the estimates \eqref{bootstrapbound} must hold for all time. 
\end{proof}
From the bootstrap estimates \eqref{bootstrapbound}, we have an  estimate for   $\gamma$ by \eqref{odegamma}: 
$$|\gamma_\tau -1|\leq\frac{400nC^2}{\gamma}\,.$$
Thus  $\frac{\gamma}{\tau}\to 1$ as $\tau\to\infty$, and we have by \eqref{def:CuCl} and \eqref{cun} that $$|(\hat{C}_u)_t+1|=|\frac{(\hat{C}_u)_t t_\tau}{\hat{C}_u}+1|=|\frac{(\hat{C}_u)_\tau}{\hat{C}_u}+1|=|\hat{c}_u+1|=|c_u|\leq\frac{n}{4\gamma}\,,\quad \tau_t=\frac{1}{\hat{C}_u}\,.$$
We can finally show that there exists a blowup time $T>0$, such that $$\lim_{t\to T}\frac{\hat{C}_u}{T-t}=1\,,\quad\lim_{t\to T}\frac{\tau}{|\log(T-t)|}=1\,.$$ %which implies $$\lim_{t\to T}\frac{\hat{C}_l}{\sqrt{(T-t)|\log(T-t)|}}= 1\,.$$ 
Moreover, defining $\kappa=\sum_i\frac{1}{\lambda_i}$, we have the following ODE $$\partial_\tau\kappa=n+4\sum_i\sum_j u_{iijj}(0)\frac{\lambda_j}{\lambda_i}\leq n+4nC\frac{G}{\gamma^2}\kappa\leq n+\frac{400nC^2}{\gamma^2}\kappa\,. $$
Therefore, for sufficiently large $\tau$, $\kappa_\tau \leq n+800nC^2\frac{\kappa}{\tau^2}$. We integrate and get for sufficiently large $\tau$, $\kappa\leq n\tau+1600n^2C^2\log\tau$. Since $\lambda=\max{\lambda_i}=1/\gamma\to1/\tau$, we have $n\leq\liminf\sum_i\frac{\lambda}{\lambda_i}\leq\limsup\sum_i\frac{\lambda}{\lambda_i}\leq n$, namely that $\lambda_i\tau\to\frac{\lambda_i}{\lambda}\to1$, and thus we arrive at the law $$\lim_{t\to T}\frac{\hat{C}_{l,i}}{\sqrt{(T-t)|\log(T-t)|}}=\lim_{t\to T}\sqrt\frac{\hat{C}_u}{{(T-t)\lambda_i\tau}}= 1\,.$$ 
Taking $C_0=100C, \lambda_0=1/(10000nC^2)$ and for initial data satisfying the assumption of Theorem \ref{thm1}, we know that $\gamma(0)=1/\lambda$ and   $u=g$ defined in the rescaled formulation satisfy the bootstrap assumption.
 From the asymptotics of $\hat{C}_u, \hat{C}_{l,i}$, we conclude Theorem \ref{thm1}.

\section{Numerical experiments}\label{sec:num}\roy{In this section, we present numerical experiments that support our analysis. In particular, we demonstrate that the normalization conditions introduced in Section~\ref{sec:dy} preserve a stable blow-up and allow us to capture the logarithmic correction numerically, both in 1D and in higher dimensions with nonradial perturbations.
}
 We remark that our proofs in the paper are derived independently of the numerical results in this section. 

\textbf{Data availability statement:} The data and the code will be available upon request.

\roy{In practice, we would like to recover the blow-up profile without assuming it a priori.
}Therefore in our numerical computation, we solve  \eqref{drf} with initial data as a large perturbation to the approximate steady state. We will compute $\hat{u}$ dynamically and recall that  our choice of normalization $\hat{c}_{l,i}$, $\hat{c}_u$ in \eqref{cun} ensures that $\hat{u}(0)$, $\nabla^2\hat{u}(0)$ remain constants in time.  

\subsection{1D case}

We choose the initialization that is more general than the assumption of our theorem as $$\hat{u}(0,z)=(1+z^2/8+z^4/10)^{-1}\,, \quad\hat{C}_u(0)=1\,, \quad\lambda\roy{(0)}=1\,.$$
At each time step $\tau_m$, we first determine the normalization constants as $$\hat{c}_u=-\hat{u}(0)-\frac{\lambda\hat{u}_{zz}(0)}{\hat{u}(0)}\,,\hat{c}_l=\frac{\hat{c}_u}{2}+\hat{u}(0)+\frac{\lambda\hat{u}_{zzzz}(0)}{2\hat{u}_{zz}(0)}\,.$$
Next, we can determine the time step $k$ via the standard numerical stability conditions for a convection-diffusion equation, and then we use the 4-th order Runge-Kutta scheme for the discretization in time and a cubic spline for the discretization in space to evolve the equation $$\hat{u}_\tau=\hat{c}_u \hat{u}-\hat{c}_l z\hat{u}_z+\hat{u}^2+\lambda\hat{u}_{zz}\,.$$
Finally, we update our $\lambda$ for time $\tau_{m+1}=\tau_m+k$ by a 4-th order Runge-Kutta discretization scheme  of the ODE $$(\log\lambda)_\tau=(2\hat{c}_l+\hat{c}_u)\,.$$

We use a fixed nonuniform mesh in space with even symmetry \roy{imposed}, and our computational domain is $[0,10^5]$ with $2000$ grid points in space. We report that after $10^9$ iterations in time, the rescaled time $\tau\approx3.9887\times10^5$ and $\log(\hat{C}_u)\approx -3.9886\times10^5$. This means that the amplitude of the solution in the physical space grows $\exp(3.9886\times10^5)$ times, which is impossible to compute if we do not use a dynamic rescaling formulation. We remark that the computation is very stable and we stopped after $10^9$ iterations only due to concerns of computational time. In theory, we can compute for an arbitrarily long time and witness an arbitrary growth of the amplitude in the physical space.

To see that the profile $\hat{u}$ indeed converges to the steady state $\bar{u}$, we plot the profile after $m=5 \times 10^4, 5 \times 10^5, 5 \times 10^6, 10^7,1.5  \times 10^7, 2 \times 10^7$ iterations and compare it with the steady state. We see that the profile converges fast; see Figure \ref{fig:profile} on the left. Furthermore, we investigate the convergence rate of the profile. Define \roy{the residual} $\gamma(\tau)=\sup_{z}\{|\hat{u}(\tau)-\bar{u}|\}$. We plot $\gamma\tau$ after $2 \times 10^7$ until $5 \times10^7$ iterations, corresponding to $\tau\in[218,11638]$. We see that the residue is approximately of order $1/\tau$; see Figure \ref{fig:profile} on the right. However, we are only using a finite domain, and as time becomes larger, the effect of the finite domain size becomes more obvious, and $\gamma\tau$ will increase slightly.

\begin{figure}
    \centering
    \includegraphics[width=5.8cm]{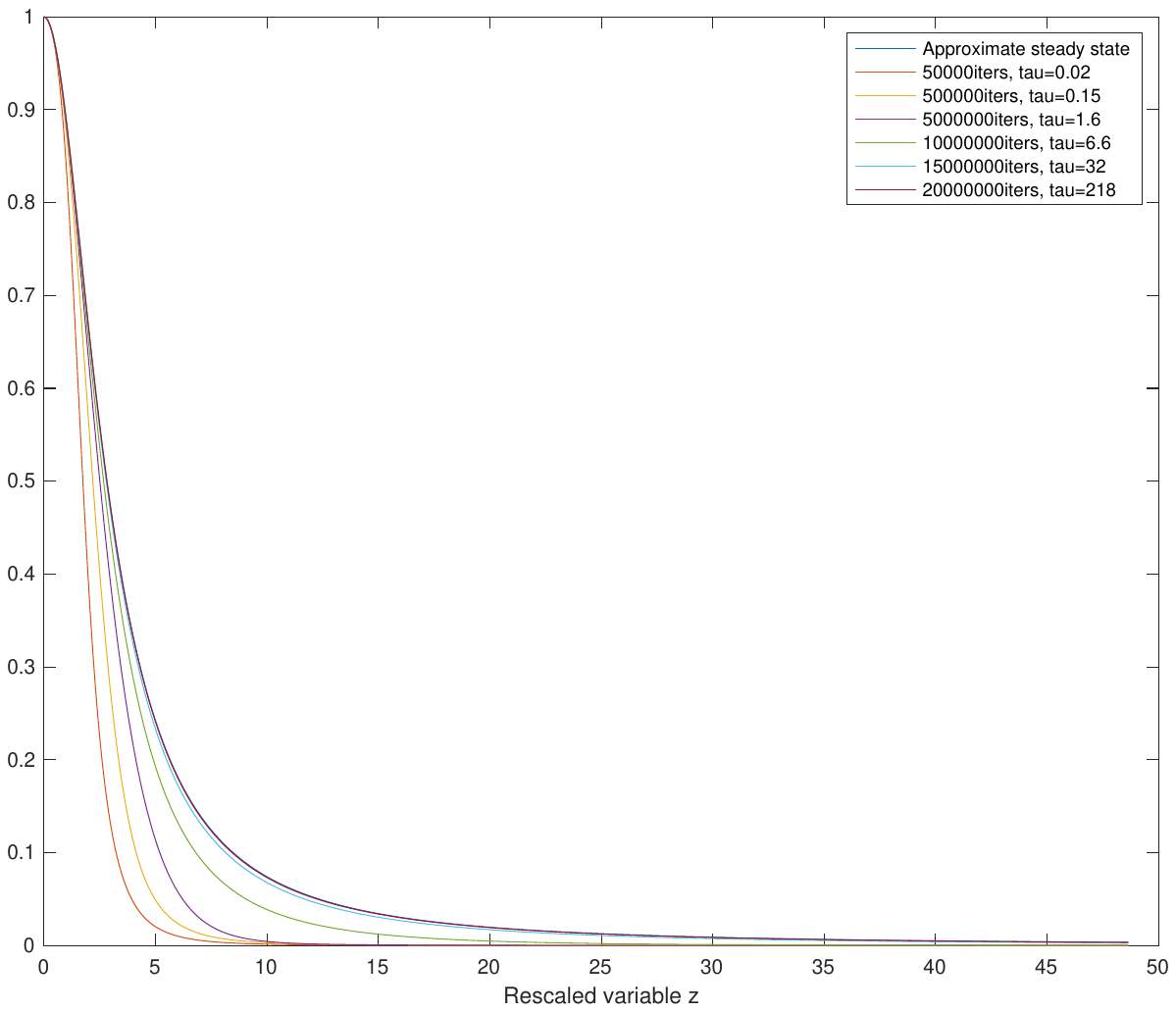}
\includegraphics[width=6cm]{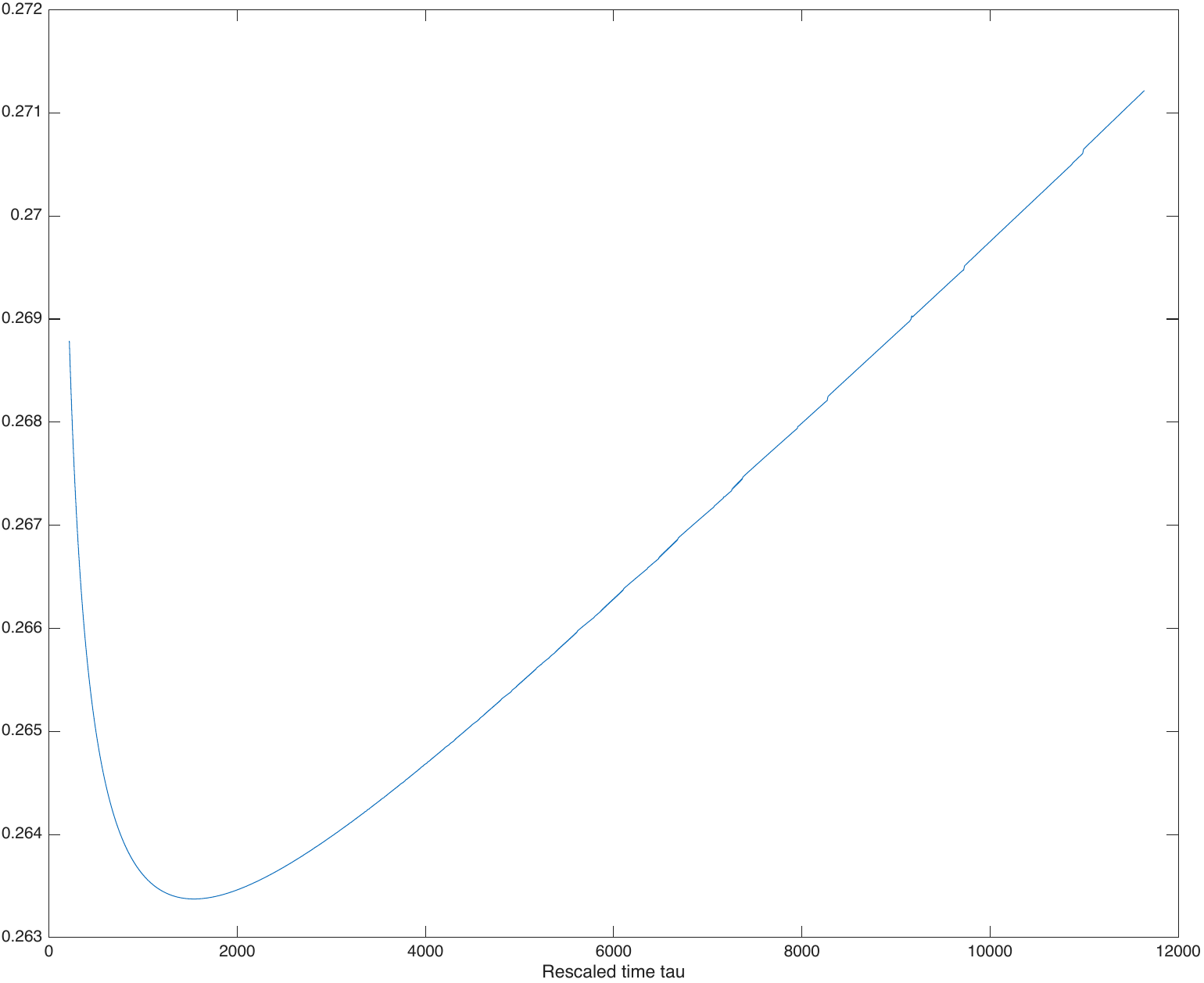}
\caption{Left: Comparison of the profile to the approximate steady state. Right: Plot of the residue multiplied by the rescaled time.}
    \label{fig:profile}
\end{figure}

To see that we can recover the correct convergence rate, we plot $(1/2-\hat{c}_l)\tau$ and $(\hat{c}_u+1)\tau$ in time to see that they indeed converge to the correct constant $5/8$ and $1/4$ respectively and therefore will give the correct log-scaling; see for example indicated by \eqref{cu}. Again, for visualization purposes, we only plot for the $2 \times 10^7$ until $5 \times 10^7$ iterations, and we can see that they converge to the desired constants very fast; see Figure \ref{fig:cl} on the left.
\begin{figure}[ht]

\centering
\includegraphics[width=6cm]{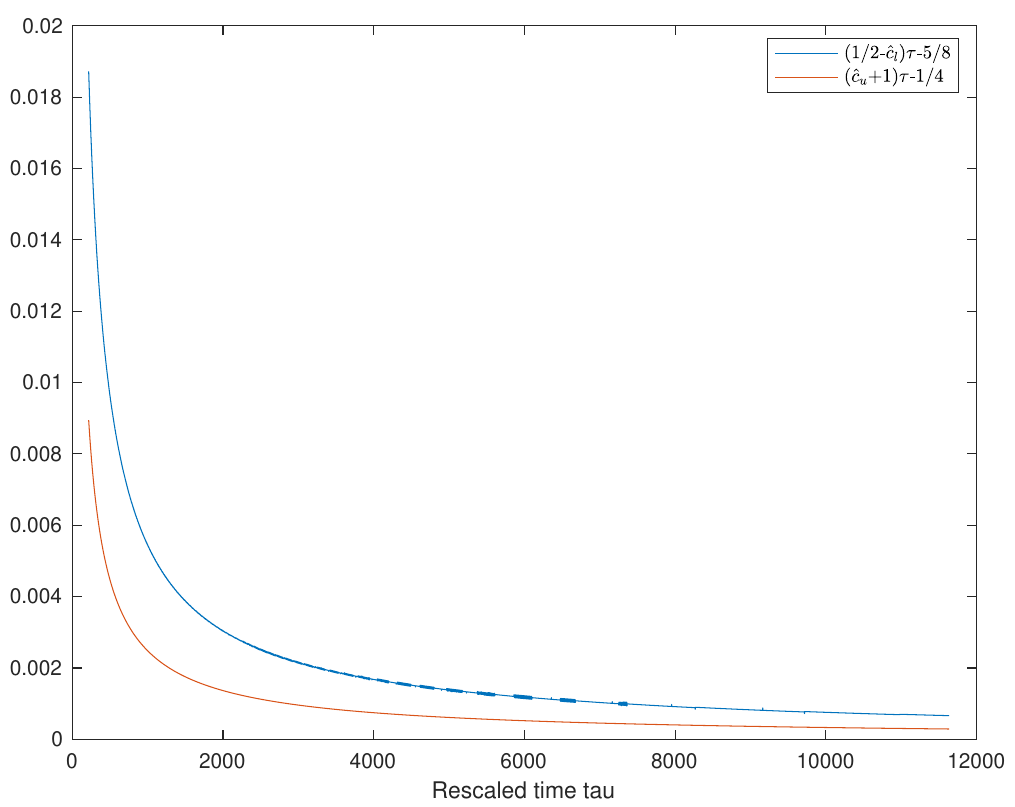}
\includegraphics[width=6cm]{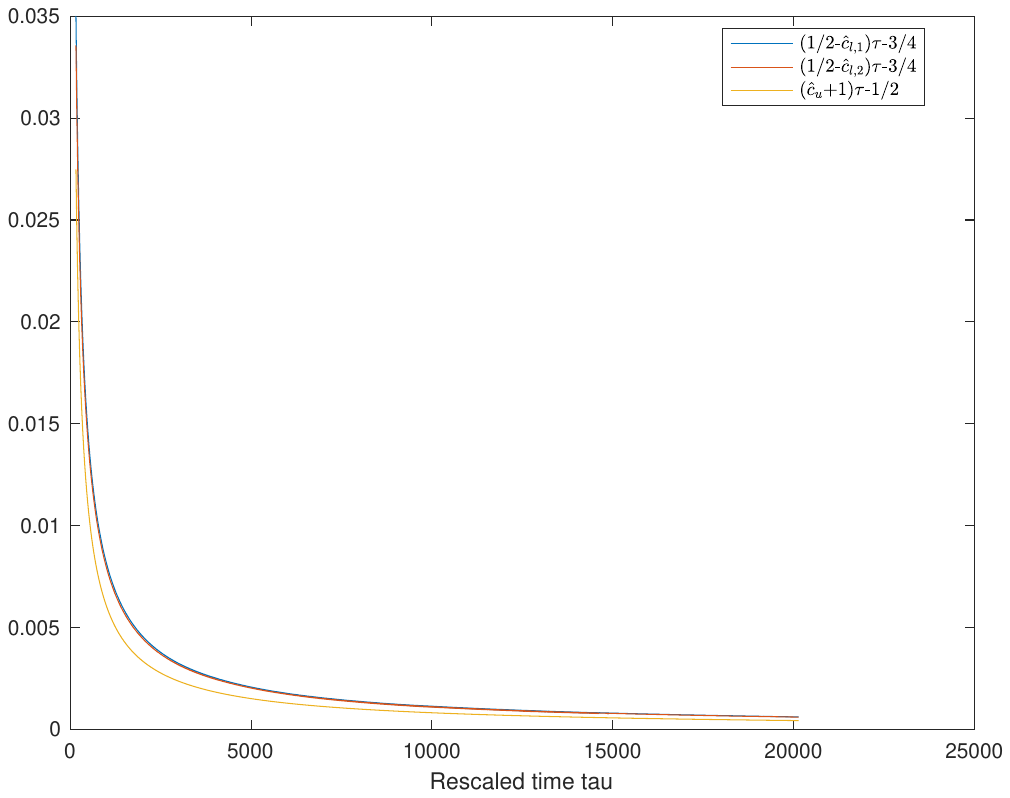}
\caption{Left: Fitting the law of the normalization constants in 1D. Right: 2D}
\label{fig:cl}
\end{figure}

\subsection{2D case}
For the 2D example, we choose a nonradial initialization  as $$\hat{u}(0,x,y)=(1+(x^2+y^2)/8+x^4/100)^{-1}\,, \quad\hat{C}_u(0)=1\,, \quad\lambda_1=\lambda_2=1\,.$$
At each time step $\tau_m$, we first determine the normalization constants as $$\hat{c}_u=-\hat{u}(0,0)-\frac{\lambda_1\hat{u}_{xx}(0,0)+\lambda_2\hat{u}_{yy}(0,0)}{\hat{u}(0,0)}\,,$$$$\hat{c}_{l,1}=\frac{\hat{c}_u}{2}+\hat{u}(0,0)+\frac{\lambda_1\hat{u}_{xxxx}(0,0)+\lambda_2\hat{u}_{xxyy}(0,0)}{2\hat{u}_{xx}(0,0)}\,,$$$$\hat{c}_{l,2}=\frac{\hat{c}_u}{2}+\hat{u}(0,0)+\frac{\lambda_1\hat{u}_{xxyy}(0,0)+\lambda_2\hat{u}_{yyyy}(0,0)}{2\hat{u}_{yy}(0,0)}\,.$$
Next, we can determine the time step $k$ via the standard numerical stability conditions for a convection-diffusion equation, and then we use the 4-th order Runge-Kutta scheme for the discretization in time and a cubic spline for the discretization in space to evolve the equation $$\hat{u}_\tau=\hat{c}_u \hat{u}-\hat{c}_{l,1} x\hat{u}_x-\hat{c}_{l,2} y\hat{u}_y+\hat{u}^2+\lambda_1\hat{u}_{xx}+\lambda_2\hat{u}_{yy}\,.$$
Finally, we update our $\lambda_1$, $\lambda_2$ for time $\tau_{m+1}=\tau_m+k$ by a 4-th order Runge-Kutta discretization scheme   of the ODE $$(\log\lambda_i)_\tau=(2\hat{c}_{l,i}+\hat{c}_u)\,,\quad i=1,2\,.$$

We use a fixed nonuniform mesh in space with even symmetry imposed, and our computational domain is $[0,4000]^2$ with $200$ gridpoints in space in each direction. 
To see that we can recover the correct convergence rate,
we plot $(1/2-\hat{c}_{l,i})\tau$ and $(\hat{c}_u+1)\tau$ as a function of $\tau$ for the $4 \times 10^6$ until $ 10^7$ iterations to see that they indeed converge to the correct constant $3/4$ and $1/2$ respectively and therefore will give the correct log-scaling; see for example indicated by \eqref{cun}. We can see that they converge to the desired constants very fast; see Figure \ref{fig:cl} on the right.

\vspace{0.2in}
{\bf Acknowledgments.} The research of T. Hou is supported by NSF Grant DMS-2205590, DMS-2508463, and the Choi Family Gift Fund. V.T. Nguyen is supported by the National Science and Technology Council (NSTC) of Taiwan. We would like to express our sincere gratitude to Changhe Yang for engaging in valuable discussions, particularly concerning the 
$n$-dimensional case, and Xiang Qin for discussions on the local well-posedness. Our thanks also go to Peicong Song for insightful conversations during the early stages of this project. Additionally, we are grateful to Dr. Jiajie Chen for his stimulating discussions. \roy{We also thank the anonymous referees for their very constructive comments and suggestions, which greatly improved the quality of the original manuscript.}

\appendix
\section{Local well-posedness} \label{sec:appendix}
In this section, we provide a proof of the local well-posedness. Namely, in equation \eqref{drf}, the local Cauchy problem for the perturbation $u=\hat{u}-\bar{u}$ is well-posed in the space $\mathcal{E}_k$ defined in \eqref{def:Ek} for {a short time}. The proof consists of two main steps: we first reduce LWP of the perturbation in the space $\mathcal{E}_k$ with a singularly weighted $L^2$-norm, to the space induced by the norm $\|\cdot\|_{\mathcal{Q}_k}=\|\cdot\|_{{{Q}_0}}+\|\cdot\|_{{{Q}_k}}$ with a regular $L^2$-norm defined in \eqref{Dj}, thanks to the vanishing conditions and interpolation inequalities. Noticing that the approximate profile lies in $\mathcal{Q}_k$ as well, so that we can work directly on the LWP of the whole solution of the original equation instead. {We use a standard fixed point argument to establish the LWP for the original problem, in particular, we shall use Duhamel's representation, the boundedness of the heat kernel and nonlinear estimates in $\mathcal{Q}_k$ to establish a contraction mapping and conclude the LWP.}

 We first aim at establishing the local well-posedness of $a$ in the original equation \eqref{sml} in ${\mathcal{Q}}_k$. We follow the argument given in \cite{BCjam96} and \cite{Weiiumj80} for the case of bounded domains.  Essentially, we use Duhamel's principle to solve the semilinear heat equation, and use the parabolic regularizing effect of the heat kernel to conclude the proof.
 \paragraph{Boundedness of the heat kernel:}

The solution to the linear heat equation with the initial data $g$ can be expressed as the convolution:
\begin{equation*}
e^{t\Delta}g=K_t\ast g\,,
\end{equation*}
where
$$K_t=\frac{1}{(4\pi t)^{n/2}} e^{-\frac{|x|^2}{4t}}\,, \quad \text{and}\,\, f \ast g(x) = \int_{\Rb^{\roy{n}}} f(y) g(x - y) dy\,.$$

We first show that the heat kernel defines a bounded operator on ${\mathcal{Q}}_k$. In fact, a {direct} calculation of the Gaussian, via boundedness of its moments, implies that $K_t$ has the following regularizing effect for any $i\in\mathbb{R}$, $0<t<1$: \begin{equation}
    \|\langle x\rangle^{|i|}K_t\|_1\lesssim 1\,.\label{bound on ek}
\end{equation}
Notic\roy{ing} that the weight satisfies $\langle x\rangle^i\lesssim\langle y\rangle^{|i|}\langle x-y\rangle^i$, {we estimate by using the boundedness of the heat kernel and a reproof of Young's inequality as:}
$$\begin{aligned}
    |\langle x\rangle^i\nabla^je^{t\Delta}g|&=\roy{|}\int \langle x\rangle^iK_t(y)\nabla^jg(x-y)dy\roy{|}\\&\lesssim(\int \langle y\rangle^{|i|}K_t(y)(\langle x-y\rangle^i\nabla^jg(x-y))^2dy)^{1/2}\|\langle y\rangle^{|i|}K_t\|^{1/2}_1\,,
\end{aligned}$$
where we use the Cauchy-Schwarz inequality. Using \eqref{bound on ek}, we estimate $$\|\langle x\rangle^i\nabla^je^{t\Delta}g\|_2^2\lesssim\int \langle y\rangle^{|i|}K_t(y)(\langle x-y\rangle^i\nabla^jg(x-y))^2dxdy\lesssim\roy{\|}\langle x\rangle^i\nabla^jg\|_2^2\,.$$
Therefore, we conclude $\|e^{t\Delta}g\|_{{\mathcal{Q}}_k}\lesssim\|g\|_{{\mathcal{Q}}_k}$.

\paragraph{LWP in the original variables:}
Identical to the $H^k$ estimates of the nonlinear terms \eqref{nonlinear hk},  we have  $$\|gf\|_{\mathcal{Q}_k}\lesssim\|g\|_{\mathcal{Q}_k}\|f\|_{\mathcal{Q}_k} \,.
$$
We thus conclude nonlinear estimates using boundedness of the heat kernel as $$\|e^{t\Delta}(gf)\|_{{\mathcal{Q}}_k}\lesssim\|gf\|_{{\mathcal{Q}}_k}\lesssim\|g\|_{{\mathcal{Q}}_k}\|f\|_{{\mathcal{Q}}_k}\,.$$
{The Duhamel's formula for \eqref{sml} reads as
$$M(a)(x,t)=e^{t\Delta}a(x,0)+\int_0^te^{(t-s)\Delta}a^2(x,s)ds\,.$$
 The above formula defines a solution map $a \mapsto M(a)$ in ${\mathcal{Q}}_k$, with the estimates 
 $$\|M(a)(\cdot,t)\|_{{\mathcal{Q}}_k}\leq C\|a(\cdot,0)\|_{{\mathcal{Q}}_k}+C\int_0^t\|a(\cdot,s)\|^2_{{\mathcal{Q}}_k}ds\,,$$
 \begin{equation}
     \|M(a)(\cdot,t)-M(b)(\cdot,t)\|_{{\mathcal{Q}}_k}\leq C\|a(\cdot,0)-b(\cdot,0)\|_{{\mathcal{Q}}_k}+ C\int_0^t\|a-b\|_{{\mathcal{Q}}_k}\|a+b\|_{{\mathcal{Q}}_k}ds\,,\label{ab dif}
 \end{equation}
 which is a contraction in the space $A$ equipped with norm $\|\cdot\|_{{\mathcal{Q}}_k,C([0,T])}$, where $$A=\{a(\cdot,0)=a_0: \|a\|_{{\mathcal{Q}}_k,C([0,T])}\leq 2C\|a_0\|_{{\mathcal{Q}}_k}\}\,,\quad\|\cdot\|_{{\mathcal{Q}}_k,C([0,T])}=\sup_{0\leq t\leq T}\|\cdot\|_{{\mathcal{Q}}_k}\,,$$ for $T<\min\{1,\frac{1}{4C^2\|a_0\|_{{\mathcal{Q}_k}}}\}$ small. We conclude the local well-posedness by a standard fixed-point argument to establish the existence and uniqueness of the fixed-point map $M(a)=a$.} 
 
 For continuous dependence, consider two initial data $a_0,b_0$. We know by the fixed point argument that for $T<\min\{1,\frac{1}{4C^2\|a_0\|_{{\mathcal{Q}_k}}},\frac{1}{4C^2\|b_0\|_{{\mathcal{Q}_k}}}\}$, the solutions $a$, $b$ have $\|\cdot\|_{{\mathcal{Q}}_k,C([0,T])}$ norms bounded by $2C\|a_0\|_{{\mathcal{Q}_k}}$, $2C\|b_0\|_{{\mathcal{Q}_k}}$ respectively. By \eqref{ab dif}, we estimate $$\|a(\cdot,t)-b(\cdot,t)\|_{{\mathcal{Q}}_k}\leq C\|a_0-b_0\|_{{\mathcal{Q}}_k}+ C\int_0^t\|a-b\|_{{\mathcal{Q}}_k}(\|a-b\|_{{\mathcal{Q}}_k}+2C\|a_0\|_{{\mathcal{Q}_k}})ds$$ and by \roy{the} Bihari–LaSalle inequality, a nonlinear generalization of Gronwall's inequality, we have that for $K=2C\|a_0\|_{{\mathcal{Q}_k}}$, $$\|a(\cdot,t)-b(\cdot,t)\|_{{\mathcal{Q}}_k}\leq\frac{KC\|a_0-b_0\|_{{\mathcal{Q}}_k}e^{CKt}}{K+C\|a_0-b_0\|_{{\mathcal{Q}}_k}-C\|a_0-b_0\|_{{\mathcal{Q}}_k}e^{CKt}}\,.$$
 \roy{Thus, we conclude the} continuous dependence for a short time.

Moreover, we establish that $a$ and $a_{ii}$ are continuously differentiable at the origin for a short time when $a$ is locally well-posed in ${\mathcal{Q}}_k$. %By \eqref{appen:inter} we establish the upper-boundedness. 
Noticing that $$a(0,t)=a(0,0)+\int_0^ta_t(0,s)ds=a(0,0)+\int_0^t(a^2(0,s)+\Delta a(0,s))ds\,,$$
$$a_{ii}(0,t)=a_{ii}(0,0)+\int_0^t(a^2(x,s)+\Delta a(x,s))_{ii}|_{x=0}ds\,,$$
%To establish the lower- and upper-bounds of $a$ and $a_{ii}$ at the origin, by triangular inequality and taking derivatives of the above formula, 
we only need to upper-bound $\|\nabla^j(a^2+\Delta a)\|_\infty$ for a short time, $j=0,2.$ It can be bounded via interpolation inequalities in Lemma \ref{lem:int} as \begin{equation}
\sup_{j\leq4}\|\nabla^ja\|_{\roy{_\infty}}+\sup_{j\leq4}\|\nabla^ja\|_{\roy{_\infty}}^2\lesssim\|a\|_{\mathcal{Q}_k}+\|a\|^2_{\mathcal{Q}_k}\,.\label{final vsi}
\end{equation}

\paragraph{Existence of solution in the regular norm:}
 Notice that the scaling parameters fix the values of $\hat{u}$, $\hat{u}_{ii}$ at the origin and recall   \eqref{32}, where one can solve the scaling parameters $\hat{C}_u$ and $\hat{C}_{l,i}$ uniquely as
\begin{equation}
    \label{resc con}\hat{C}_u(\tau)=\hat{C}_u(0)a(0,0)/a(0,t)\,,\quad \hat{C}_{l,i}(\tau)=\sqrt{\frac{a_{ii}(0,0)a(0,t)}{a_{ii}(0,t)a(0,0)}}\,,
\end{equation}
 which remain continuously differentiable and upper- and lower-bounded for a short time, as long as $a$ and $a_{ii}$ are continuously differentiable and hence upper- and lower-bounded for a short time at the origin.\footnote{Recall \eqref{def:renormua} that $\tau$ is related to $t$ via $dt/d\tau=\hat{C}_u$, which is well-defined for a short time.} Since $\hat{u}$ is just a rescaling of $a$ via \eqref{scaling:nd}, the existence and continuity in $\mathcal{Q}_k$ of $\hat{u}$ follow from that of $a$ and the boundedness and continuous differentiability of the scaling parameters.  Since $\bar{u}$ is in $\mathcal{Q}_k$, we know that $u=\hat{u}-\bar{u}$ is in $\mathcal{Q}_k$. Finally, we show that $\hat{u}$ defined via our rescalings \eqref{scaling:nd} satisfies the dynamic rescaling equation \eqref{drf}, which is true since the rescaling parameters $\hat{C}_u, \hat{C}_{l,i}$ defined above are continuously differentiable in time, thanks again to the regularity of $a$.

 \paragraph{Existence of solution in the singular norm:} Recall that our modulation conditions fix $\hat{u}$, $\hat{u}_{ii}$ at the origin, guaranteeing that the even perturbation $u=O(|z|^4)$ near the origin. We will show that its $L^2_\rho$-norm can be bounded by $\mathcal{Q}_k$. 
Similar to $L^2$ estimates of the viscous terms  in Section \ref{subl2}, we estimate
\begin{equation*}
    \begin{aligned}
\|u\|_{\rho}^2&\lesssim\int_{|z|\leq 1}u^2|z|^{-5-n}+\int_{|z|\geq 1}u^2\langle z\rangle^{1-n}\lesssim(\sup_I|u||z|^{-3})^2+\|u\|_{Q_0}^2\\&\lesssim\|\nabla^3u\|_\infty^2+\|u\|_{Q_0}^2\lesssim\|u\|_{\mathcal{Q}_k}^2\,,\end{aligned}
\end{equation*}
 where \roy{we have used \eqref{final vsi}} in the last inequality. We conclude \begin{equation}
\label{reg}\|u\|_{\mathcal{E}_k}\lesssim\|u\|_{\mathcal{Q}_k}\,,\end{equation} \roy{which implies} that $u$ exists and is continuous in time in ${\mathcal{E}}_k$ for a short time.

 \paragraph{Conclusion of the LWP proof:}

Finally, we establish uniqueness and continuous dependence on initial data. Suppose $\hat{u}^1,\hat{u}^2\in {\mathcal{E}}_k$ solves \eqref{drf} for the same initial data and that $\hat{u}^j-\bar{u}=O(|z|^4)$ near the origin for $j=1,2$. In particular, by interpolation inequality in Lemma \ref{lem:int}, they are both in $C^4$. Define $$\hat{C}_{u}^j(\tau)=\lambda_{1}^j(0)\exp{\Big(\int_0^\tau \hat{c}_{u}^j(s) ds\Big)},\quad\hat{C}_{l,i}^j(\tau)=\exp{\Big(\int_0^\tau -\hat{c}_{l,i}^j(s) ds\Big)}\,,$$
$$t^j(0)=0,\quad \frac{dt^j}{d\tau}=\hat{C}_{u}^j,\quad a^j(x,t)=(\hat{C}_{u}^j)^{-1}\hat{u}^j(x_1/\hat{C}_{l,1}^j,x_2/\hat{C}_{l,2}^j,\cdots,x_n/\hat{C}_{l,n}^j,\tau^j(t))$$
for $j=1,2$. Here we use the functions $t^j(\tau)$ and $\tau^j(t)$ to denote the time rescaling. \roy{Under the inverse transformation,} $a^j$ satisfies the semilinear heat equation in the original variables \eqref{sml}. By the LWP established previously for the original semilinear heat equation, we know that $a^1=a^2$. By the vanishing conditions $\hat{u}^j-\bar{u}=O(|z|^4)$, we deduce \roy{from} \eqref{resc con} that $$\hat{C}_{u}^1(\tau^1(t))=\hat{C}_{u}^2(\tau^2(t)),\quad \hat{C}_{l,i}^1(\tau^1(t))=\hat{C}_{l,i}^2(\tau^2(t))\,.$$ Finally we have $\frac{d\tau^1}{dt}=(\hat{C}_{u}^1(\tau^1(t)))^{-1}=(\hat{C}_{u}^2(\tau^2(t)))^{-1}=\frac{d\tau^2}{dt}$, hence $\tau^1=\tau^2$ and we conclude $\hat{u}^1=\hat{u}^2$.

Continuous dependence is  standard \roy{once the uniqueness is established}, since we have already established the continuous dependence of the solution $a$ in the physical variable of the original equation, and the rescalings from $a$ to $\hat{u}$ are continuously dependent on $a$. Now we have the continuous dependence of $u$ in ${{\mathcal{Q}}}_k$, and by \eqref{reg} and the trivial bound $\|\cdot\|_{{{\mathcal{Q}}}_k}\leq\|\cdot\|_{{{\mathcal{E}}}_k}$ via definition \eqref{Dj}, in ${{\mathcal{E}}}_k$.

\vspace{0.2in}
\noindent {\bf Conflict of interest.} The authors declare no conflict of interest.

\bibliographystyle{abbrv} 
\bibliography{ref}
\end{document}